\documentclass[11pt]{amsart}
\usepackage{amsmath}
\usepackage{amsthm}
\usepackage{amssymb}
\usepackage{amsfonts,mathrsfs}
\usepackage{stmaryrd}
\usepackage{amsxtra}  
\usepackage{epsfig}
\usepackage{verbatim}

\theoremstyle{plain}
\newtheorem{theorem}[equation]{Theorem}
\newtheorem{proposition}[equation]{Proposition}
\newtheorem{lemma}[equation]{Lemma}
\newtheorem{corollary}[equation]{Corollary}
\newtheorem{definition}[equation]{Definition}
\theoremstyle{remark}
\newtheorem*{remark}{Remark}
\numberwithin{equation}{section}

\newcommand{\sjump}{\hskip .2 cm}

\newcommand{\dbar}{\bar \partial}

\begin{document}

\title{A note on a smoothing property of the harmonic Bergman projection}
\author{A.-K. Herbig}
\thanks{Research supported by Austrian Science Fund FWF grant  V187N13.}

\address{Department of Mathematics, \newline University of Vienna, Vienna, Austria}
\email{anne-katrin.herbig@univie.ac.at}
\begin{abstract}
It is proved that on any smoothly bounded domain $\Omega\Subset\mathbb{R}^{n}$, $n\geq2$, the output of the harmonic Bergman projection
belongs to the Sobolev space of order $k\in\mathbb{N}$ as long as all tangential derivatives of order up to $k$ of the input function belong to $L^{2}(\Omega)$.
\end{abstract}
\date{\today}
\maketitle

\section{Introduction}\label{S:Intro}
Let $\Omega\Subset\mathbb{R}^{n}$, $n\geq 2$, be a smoothly bounded domain. Denote  by $h^{0}(\Omega)$ the space of functions which are harmonic on $\Omega$ and belong to $L^{2}(\Omega)$. The harmonic Bergman projection~$P$  is the orthogonal projection of $L^{2}(\Omega)$ onto its closed subspace $h^{0}(\Omega)$.

The harmonic Bergman projection is known to be a continuous map of $H^{k}(\Omega)$, the $L^{2}(\Omega)$-Sobolev space of order $k\in\mathbb{N}$, to itself. That is, for all $k\in\mathbb{N}$ there exists a constant $C_{k}>0$ such that
\begin{align}\label{E:exactreg}
  \left\|Pf\right\|_{k}\leq C_{k}\left\|f\right\|_{k}\qquad\sjump\forall\sjump f\in H^{k}(\Omega),
\end{align}
where $\|.\|_{k}$ is the $L^{2}(\Omega)$-Sobolev norm of order $k$.
The purpose of this note is to show that $P$ satisfies a stronger estimate, where the right hand side of \eqref{E:exactreg} may be substituted by an $L^2(\Omega)$-norm only measuring tangential derivatives. 

Call the collection $\mathcal{T}=\{T_{j}\}_{j=1}^m$, $m\in\mathbb{N}$,  of vector fields with coefficients in $\mathcal{C}^{\infty}(\overline{\Omega})$ a tangential spanning set for $\Omega$ if it spans the tangent space to the boundary, $b\Omega$, of $\Omega$ at each boundary point. Denote by $\|.\|_{k,\mathcal{T}}$ the associated tangential $L^2(\Omega)$-Sobolev norm of order $k\in\mathbb{N}$ and by $H_{\mathcal{T}}^{k}(\Omega)$ the corresponding Hilbert space, see Definition \ref{D:tangSobolev} and the subsequent paragraph.

\begin{theorem}\label{T:main}
  Let $\Omega\Subset\mathbb{R}^{n}$, $n\geq 2$, be a smoothly bounded domain, $k\in\mathbb{N}$. Let $\mathcal{T}$ be a tangential spanning set for $\Omega$. Then there exists a constant $c_{k}>0$ such that
  \begin{align}\label{E:main}
    \left\|Pf\right\|_{k}\leq c_{k}\|f\|_{k,\mathcal{T}}\qquad\sjump\forall\sjump f\in H_{\mathcal{T}}^{k}(\Omega).
  \end{align}
\end{theorem}
It follows from the Sobolev embedding theorem that $P$ does not just exhibit a smoothing behavior in the $L^{2}(\Omega)$-Sobolev scale, but in fact maps $\bigcap_{k=0}^{\infty}H_{\mathcal{T}}^{k}(\Omega)$, a class of functions strictly larger than $\mathcal{C}^{\infty}(\overline{\Omega})$ (see end of Section \ref{S:Prelims}), to $\mathcal{C}^{\infty}(\overline{\Omega})$.

\begin{corollary}
  Let $\Omega\Subset\mathbb{R}^{n}$, $n\geq 2$, be a smoothly bounded domain. Let $\mathcal{T}$ be a tangential spanning set for $\Omega$. Then $P$ maps
  $\bigcap_{k=0}^{\infty}H_{\mathcal{T}}^{k}(\Omega)$ continuously to
 $\mathcal{C}^{\infty}(\overline{\Omega})$.
\end{corollary}

The analytic Bergman projection, $B$, which projects $L^{2}(\Omega)$ orthogonally onto the space of holomorphic functions in $L^{2}(\Omega)$, satisfies a similar smoothing property \cite{HerMcN10,HerMcNStr11}. In fact, this note is a natural continuation of \cite{HerMcNStr11} and the proof of Theorem \ref{T:main} is based on the proof of Theorem 1.1 in \cite{HerMcNStr11}. That the latter proof is extendible to the situation at hand is essentially due to both $B$ and $P$ being projections onto the kernel of an elliptic differential operator, $\dbar$ and $\Delta$, respectively. Theorem \ref{T:main} and its proof may serve as a prototype for the analysis of smoothing properties for other orthogonal projections onto kernels of elliptic differential operators of order higher than $1$.

The article is structured as follows. In Section \ref{S:Prelims} basic notions and facts are reviewed. The proof of Theorem \ref{T:main} is given in Section \ref{S:Proof}. Section \ref{S:Ball} contains a direct proof of Theorem \ref{T:main} for a particular choice of $\mathcal{T}$ when $\Omega$ is the unit ball using the explicit representation of the harmonic Bergman kernel.  The section concludes with an example on the unit disk  which shows that the right hand side of \eqref{E:main} may not be substituted by a norm measuring only normal derivatives.

\section{Preliminaries}\label{S:Prelims}

\subsection{Full and tangential $L^{2}$-Sobolev spaces}

Let $\Omega\Subset\mathbb{R}^{n}$, $n\geq 2$, be a bounded domain with smooth boundary $b\Omega$. 
Write $\mathcal{C}^{\infty}(\Omega)$, $\mathcal{C}^{\infty}(\overline{\Omega})$ and $\mathcal{C}^{\infty}_{0}(\Omega)$ for the spaces of functions which are smooth in $\Omega$, smooth up to $b\Omega$ and smooth with compact support in $\Omega$, respectively.
The $L^{2}(\Omega)$-inner product for real-valued functions $f$, $g$ on $\Omega$  is $$(f,g)=\int_{\Omega}fg\;dV,$$
where $dV$ is the Euclidean volume form. Let $\|.\|=(.,.)^{1/2}$ be the induced $L^{2}(\Omega)$-norm on $\Omega$.
For $\alpha=(\alpha_{1},\dots,\alpha_{n})\in\mathbb{N}^{n}_{0}$ a multi-index of length $|\alpha|=\sum_{j=1}^{n}\alpha_{j}$ set $$D^{\alpha}=\frac{\partial^{|\alpha|}}{\partial x^{\alpha_{1}}_{1}\dots \,\partial x^{\alpha_{n}}_{n}}\,.$$ For $k\in\mathbb{N}$, the Sobolev space $H^{k}(\Omega)$ of order $k$ for functions on $\Omega$ is defined to be
\begin{align*}
  \Bigl\{f\in L^{2}(\Omega): D^{\alpha}f\in L^{2}(\Omega)\sjump\sjump\forall\sjump\alpha\;\text{with}\;|\alpha|\leq k
  \Bigr\},
\end{align*}
where $D^{\alpha}f$ is taken in the sense of distributions. $H^{k}(\Omega)$ equipped with the inner product
$$(f,g)_{k}:=\sum_{|\alpha|\leq k}(D^{\alpha}f,D^{\alpha}g)\qquad\sjump\forall\sjump f,g\in H^{k}(\Omega)$$
is a Hilbert space  and $C^{\infty}(\overline{\Omega})$ is dense with respect to the induced norm $\|.\|_{k}$.
Denote by $H_{0}^{k}(\Omega)$ the closure of $\mathcal{C}^{\infty}_{0}(\Omega)$ with respect to $\|.\|_{k}$. The Sobolev space $H^{-k}(\Omega)$ of order $-k$ is defined to be the dual of $H_{0}^{k}(\Omega)$, and hence is endowed with the norm given by the operator norm. 

\medskip

To define tangential Sobolev norms on $\Omega$ without having to resort to local coordinates let us employ smooth vector fields on $\overline{\Omega}$ which are tangential to $b\Omega$ as follows.
\begin{definition}
 A set $\mathcal{T}=\{T_{j}\}_{j=1}^{m}$ of  vector fields $T_{j}$ with coefficients in $\mathcal{C}^{\infty}(\overline{\Omega})$ is said to be a tangential spanning set for $\Omega$ if for all $x\in b\Omega$ the span of $T_{1}(x),\dots, T_{m}(x)$ equals the tangent space $T_{x}(b\Omega)$ of $b\Omega $ at $x$.
\end{definition}
Note that if $N$ is a smooth vector field in a neighborhood of $\overline{\Omega}$ which is normal to $b\Omega$ at each boundary point, then it follows that the span of $T_{1}(x),\dots,T_{m}(x)$ and $N(x)$ does not only equal $\mathbb{R}^{n}$ when $x\in b\Omega$ but also when $x$ varies over a sufficiently small neighborhood of $b\Omega$.
\begin{definition}\label{D:tangSobolev}
Let $\mathcal{T}=\{T_{j}\}_{j=1}^{m}$ be a tangential spanning set for a smoothly bounded domain $\Omega\Subset\mathbb{R}^{n}$, $n\geq 2$.

\noindent(i) For $\ell\in\mathbb{N}$, define $J_{\ell}=(j_{1},\dots,j_{\ell})$ with $j_{i}\in\{1,\dots,m\}$ for all $i\in\{1,\dots,\ell\}$and the differential operator $T_{J_{\ell}}$ of order $\ell$ by setting
     $$T_{J_{\ell}}=T_{j_{1}}\circ\dots\circ T_{j_{\ell}}.$$
When $\ell=0$, set $J_{\ell}=0$ and $T_{0}$ the identity map.

\noindent(ii)  For $k\in\mathbb{N}$, the tangential Sobolev space $H_{\mathcal{T}}^{k}(\Omega)$ of order $k$ with respect to the tangential spanning set $\mathcal{T}$ is 
\begin{align*}
   \Bigl\{f\in  L^{2}(\Omega): T_{J_{\ell}}f\in L^{2}(\Omega)\qquad\sjump\forall\sjump\ell\in\{1,\dots,k\}
   \Bigr\},
\end{align*}
where $T_{J_{\ell}}f$ is taken in the distributional sense.
\end{definition}
$H_{\mathcal{T}}^{k}(\Omega)$ with the  $L^{2}(\Omega)$-inner product $$(f,g)_{k,\mathcal{T}}=\sum_{\ell=0}^{k}\sum_{J_{\ell}}
(T_{J_{\ell}}f,T_{J_{\ell}}g)\qquad\sjump\forall\sjump f,g\in H_{\mathcal{T}}^{k}(\Omega)$$ 
is a Hilbert space and $C^{\infty}(\overline{\Omega})$ lies densely in it with respect to the induced norm $\|.\|_{k,\mathcal{T}}$. Although, two tangential spanning sets $\mathcal{T}_{1}$ and $\mathcal{T}_{2}$ span the same sets on the boundary of a given domain, the induced norms $\|.\|_{k,\mathcal{T}_{1}}$ and $\|.\|_{k,\mathcal{T}_{2}}$ are in general not equivalent, see Section 5 in \cite{HerMcN10} for examples. Furthermore, the Fr\'{e}chet space $H_{\mathcal{T}}^{\infty}(\Omega):=\bigcap_{k=0}^{\infty}H_{\mathcal{T}}^{k}(\Omega)$ contains more functions than $C^{\infty}(\overline{\Omega})$, see also  Section 5 in \cite{HerMcN10} for examples.

\subsection{Harmonic functions and the harmonic Bergman projection}

For $k\in\mathbb{Z}$, denote by $h^{k}(\Omega)$ the space of harmonic functions which belong to $H^{k}(\Omega)$.
It follows essentially from the mean value property for harmonic functions that $h^{0}(\Omega)$ is a closed subspace, so that the harmonic Bergman projection $$P: L^{2}(\Omega)\rightarrow h^{0}(\Omega),$$ orthogonally projecting functions in $L^{2}(\Omega)$ onto $h^{0}(\Omega)$, is defined.  Furthermore, it is known that the harmonic Bergman projection is a bounded operator from $H^{k}(\Omega)$ to itself for $k\in\mathbb{N}_{0}$, i.e., \eqref{E:exactreg} holds, see \cite{Bell82}, the beginning of the proof of Theorem 1 therein. 

Inequality \eqref{E:main} might appear contradictory to the fact that $P$ is the identity map on $h^{0}(\Omega)$. However, this discrepancy is resolved, e.g., for $k=1$, after realizing that the ellipticity of $\triangle$ implies that for any tangential spanning set $\mathcal{T}$ there exists a constant $C>0$ such that
\begin{align}\label{E:harmtang}
  \|h\|_{1}\leq C\|h\|_{1,\mathcal{T}}\qquad\sjump\forall\sjump h\in h^{0}(\Omega)\cap H_{\mathcal{T}}^{1}(\Omega),
\end{align}
a proof of \eqref{E:harmtang} may be derived from the proof of Lemma 2.1 in \cite{Kohn99}. An estimate similar to \eqref{E:harmtang} also holds when $\mathcal{T}$ is replaced by a smooth vector field $N$  transversal to $b\Omega$, see \cite{Det81}, i.e., there is a constant $\widetilde{C}>0$ such that
\begin{align*}
  \|h\|_{1}\leq \widetilde{C}\Bigl(\|Nh\|+\|h\| \Bigr)\sjump\qquad\forall h\in h^{0}(\Omega)\;\;\text{with}\;\;Nh\in L^{2}(\Omega).
\end{align*} 
Nevertheless, smoothing by $P$ in tangential directions does not hold in general, see the second part of Section~\ref{S:Ball}.

Ligocka derived in \cite{Lig87}, see Theorem 3, that for $k\in\mathbb{N}$, the space $h^{-k}(\Omega)$ is equal to the space of harmonic functions equipped with the $L^{2}(\Omega)$-norm weighted with $(-r)^{2k}$, where $r$ is some smooth defining function for $\Omega$. In particular, the part of norm equivalence of interest here  may be stated as follows: there exists constants $c_{k}$ such that
\begin{align}\label{E:harmnegnorm}
  \|r^{k}h\|\leq c_{k}\|h\|_{-k}\qquad\sjump\forall\sjump h\in h^{-k}(\Omega).
\end{align}

Next, write $h^{\infty}(\Omega)$ and $h^{-\infty}(\Omega)$ for $\bigcap_{j=0}^{\infty}h^{j}(\Omega)$ and $\bigcap_{j=0}^{\infty}h^{-j}(\Omega)$, respectively. Bell showed in Theorem 1 in \cite{Bell82}  that the latter two spaces are mutually dual to each other. Ligocka further developed this theme and showed that $h^{k}(\Omega)$ and $h^{-k}(\Omega)$ are mutually dual, see Theorem 2 in \cite{Lig87}. As a consequence of Ligocka's work, analyzing the Sobolev-$k$-norm of the harmonic Bergman projection acting on a function $f$ reduces to considering  $L^{2}(\Omega)$-pairings of $f$ with elements of the unit ball in $h^{-k}(\Omega)$ as follows.

\begin{corollary}\label{C:dualityredux}
  Let $\Omega\Subset\mathbb{R}^{n}$, $n\geq 2$, be a smoothly bounded domain, $k\in\mathbb{N}$. Then there exists a constant $C_{k}>0$ such that $$\|Pf\|_{k}\leq C_{k}\sup\{|(f,h)| : h\in h^{k}(\Omega), \|h\|_{-k}\leq 1 \}$$
 for all $f\in h^{k}(\Omega)$. 
\end{corollary}
A direct proof of Corollary \ref{C:dualityredux} may also be derived as for the analogous statement for the analytic Bergman projection in \cite{HerMcNStr11}, see Remark 2.8 and Proposition 2.3 therein.

\subsection{Normal antiderivatives and their estimates}\label{SS:Flowstuff}

This section is a review on how to construct antiderivatives (with estimates) along  integral curves associated to a vector field normal to $b\Omega$. In fact, this constitutes  a summary of Section 4 in \cite{HerMcNStr11}, see the later for further details and proofs.

Let $N=\sum_{j=1}^{n}N_{j}\frac{\partial}{\partial x_{j}}$ be a vector field whose coefficients, $N_{j}$, $j\in\{1,\dots,n\}$,  are smooth in a neighborhood of $\overline{\Omega}$. Suppose that $N$ is transversal to $b\Omega$.  Then there exist a scalar $\tau_{0}>0$, a neighborhood $U$ of $b\Omega$ and a map $\varphi:(-\tau_{0},\tau_{0})\times U\rightarrow\mathbb{R}^{n}$ such that
\begin{itemize}
  \item[(a)] $\varphi(0,x)=x$ for all $x\in U$,
  \item[(b)] for all $\ell\in\{1,\dots,n\}$ and $(t,x)\in(-\tau_{0},\tau_{0})\times U$ $$\frac{\partial\varphi_{\ell}}{\partial t}(t,x)=N_{\ell}(\varphi(t,x)).$$
\end{itemize}
Moreover, for each $x\in U$, $\varphi(.,x)$ is a diffeomorphism from $(-\tau_{0},\tau_{0})$ to the curve $\{\varphi(t,x):t\in(-\tau_{0},\tau_{0})\}$. This fact together with the transversality of $N$ implies that
for each $x\in U$ there exists a unique scalar $t_{x}$ for which $\varphi(t_{x},x)\in b\Omega$ holds. Note that it may be assumed that $t_{x}>0$ for $x\in\Omega\cap U$, otherwise replace $N$ by $-N$. Furthermore, after possibly rescaling (of $N$), it may be assumed that $\tau_{0}=1$. Denote by $\mathcal{C}_{\overline{U}}^{\infty}(\Omega)$ the space of functions belonging to $\mathcal{C}^{\infty}(\Omega)$ which are identically zero on $\Omega\setminus\overline{U}$. Define the operator
$\mathfrak{A}:\mathcal{C}_{\overline{U}}^{\infty}(\Omega)\rightarrow\mathcal{C}_{\overline{U}}^{\infty}(\Omega)$ by
\begin{align*}
  \mathfrak{A}[g](x)=\int_{-1}^{0}(g\circ\varphi)(s,x)\;ds\qquad\sjump\forall\sjump x\in\Omega\cap U,
\end{align*}
and $\mathfrak{A}[g](x)=0$ when $x\in\Omega\setminus U$. It then follows from the Fundamental Theorem of Calculus (see also Lemma 4.2 in \cite{HerMcNStr11}), that for $g\in\mathcal{C}_{\overline{U}}^{\infty}(\Omega)$ 
\begin{align}\label{E:FTC}  
  g(x)=\mathfrak{A}[Ng](x)\qquad\sjump\forall\sjump x\in\Omega.
\end{align}

To organize operators generated by compositions of  $\mathfrak{A}$, differential operators, and their commutators, first introduce the following spaces.

\begin{definition}
  (1)  An operator $A: \mathcal{C}_{\overline{U}}^{\infty}(\Omega)\rightarrow \mathcal{C}_{\overline{U}}^{\infty}(\Omega)$  is said to belong to $ 
  \mathcal{A}_{\mu,0}^{1}$ for $\mu\in\mathbb{N}_{0}$ if there is a function $\gamma\in\mathcal{C}^{\infty}([-1,0]\times U)$ such that
   $$A[g](x)=\int_{-1}^{0}s^{\mu}\gamma(s,x)\cdot(g\circ\varphi)(s,x)\;ds\qquad\sjump\forall\sjump x\in\Omega\cap U,$$
  and $A[g](x)=0$ for $x\in\Omega\setminus U$.
  
  \noindent(2) An operator $A: \mathcal{C}_{\overline{U}}^{\infty}(\Omega)\rightarrow \mathcal{C}_{\overline{U}}^{\infty}(\Omega)$ is said to belong to $\mathcal{A}_{\alpha,0}^{\ell}$ for $\ell\in\mathbb{N}$ and $\alpha=(\alpha_{1},\dots,\alpha_{\ell})\in\mathbb{N}_{0}^{\ell}$ if it belongs to $$\text{span}\left(A_{1}\circ\dots\circ A_{\ell}:A_{j}\in\mathcal{A}_{\alpha_{j},0}^{1} \right).$$
  
  \noindent(3) An operator $A: \mathcal{C}_{\overline{U}}^{\infty}(\Omega)\rightarrow \mathcal{C}_{\overline{U}}^{\infty}(\Omega)$ is said to belong to $\mathcal{A}_{\alpha,\nu}^{\ell}$ for $\ell\in\mathbb{N}$, $\alpha\in\mathbb{N}_{0}^{\ell}$ and $\nu\in\mathbb{N}$ if it belongs to the 
  $$\text{span}\left(A_{\alpha}^{\ell}\circ D^{\beta}: A_{\alpha}^{\ell}\in\mathcal{A}_{\alpha,0}^{\ell},\;|\beta|\leq\nu\right).$$
\end{definition}
The following lemma clarifies the graded structure of $\mathcal{A}_{*,*}^{*}$.
\begin{lemma}\label{L:ADcommutator}
 (i)  If $A\in\mathcal{A}_{\alpha,\nu}^{\ell}$ for some $\ell\in\mathbb{N}$, $\alpha\in\mathbb{N}_{0}^{\ell}$ and $\nu\in\mathbb{N}_{0}$, then
  $$[A,D^{\beta}]\in\mathcal{A}_{\alpha,\nu+|\beta|-1}^{\ell}+\sum_{j=1}^{\ell}\mathcal{A}_{\alpha+e_{j},\nu+|\beta|}^{\ell},$$
  where $e_{j}$ is the standard $j$-th unit vector.
  
  \noindent(ii) If $A_{j}\in\mathcal{A}_{\alpha_{j},\nu_{j}}^{\ell_{j}}$ for some $\ell_{j}\in\mathbb{N}$, $\alpha_{j}\in\mathbb{N}_{0}^{\ell}$ and $\nu_{j}\in \mathbb{N}_{0}$ for $j=1,2$, then $$A_{1}\circ A_{2}\in \mathcal{A}_{(\alpha_{1},\alpha_{2}),\nu_{1}+\nu_{2}}^{\ell_{1}+\ell_{2}}.$$
\end{lemma}
\begin{proof}
Part (i) is Lemma 4.25 in \cite{HerMcNStr11}. Part (ii) follows straightforwardly from part (i).
\end{proof} 

Eventually we will be interested in the operators in $\mathcal{A}_{*,*}^{*}$ as operators on $h^{k}(\Omega)$. To derive mapping properties of these operators in the $L^{2}$-Sobolev scale, the following classes of operators are introduced.

\begin{definition}
  An operator $A: \mathcal{C}_{\overline{U}}^{\infty}(\Omega)\rightarrow \mathcal{C}_{\overline{U}}^{\infty}(\Omega)$ is said to belong to $\mathcal{S}_{\nu}^{k}$ for $\nu, k\in\mathbb{N}_{0}$ if there is a constant $C>0$ such that
  \begin{align*}
    \left\|t_{x}^{\ell}\cdot A[g] \right\|\leq C\sum_{|\beta|\leq\nu}\left\|t_{x}^{\ell+k}D^{\beta}g\right\|\qquad\sjump\forall\sjump\ell\in\mathbb{N}_{0},
    \sjump g\in\mathcal{C}_{\overline{U}}^{\infty}(\Omega),
  \end{align*}
  where $C$ does not depend on $g$ or $\ell$.
\end{definition}

It follows from Hardy's inequality that
\begin{align}\label{L:AS}
 \mathcal{A}_{\alpha,\nu}^{\ell}\subset\mathcal{S}_{\nu}^{\ell+\alpha};
\end{align}
for a proof see Lemma 4.8 in \cite{HerMcNStr11}.

\medskip

Throughout, for $A, B\in\mathbb{R}$ non-negative, write $A\lesssim B$ when $A\leq c B$ holds for some constant $c>0$.


\section{The proof of Theorem \ref{T:main}}\label{S:Proof}

The proof of Theorem \ref{T:main} is based on representing a given harmonic function as a linear combination of $T_{j}$-derivatives up to the order $k$ of certain functions with ``good'' $L^{2}(\Omega)$-control in terms of the given data. 
\begin{proposition}\label{P:antiderivative}
  Let $\Omega\Subset\mathbb{R}^{n}$, $n\geq 2$, be a smoothly bounded domain, $k\in\mathbb{N}$. Let $\mathcal{T}=\{T_{j}\}_{j=1}^{m}$ be a tangential spanning set for $\Omega$. Then there exist a neighborhood $U$ of $b\Omega$, a function $\zeta\in\mathcal{C}_{0}^{\infty}(\overline{\Omega}\cap U)$ which equals $1$ near $b\Omega$  and constants $C_{k}>0$ such that for all $h\in h^{0}(\Omega)$ there 
 are functions $\mathcal{H}_{J_{\ell}}^{k}\in H^{k}(\Omega)\cap C^{\infty}(\Omega)$ for $\ell\in\{0,\dots,k\}$ satisfying
 \begin{itemize}
   \item[(1)] $\zeta h=\sum_{\ell=0}^{k}\sum_{J_{\ell}}T_{J_{\ell}}(\mathcal{H}_{J_{\ell}}^{k})$ on $\Omega$,
   \medskip
   \item[(2)] $\|\mathcal{H}_{J_{\ell}}^{k}\|\leq C_{k}\|h\|_{-k}$ for all $\ell\in\{0,\dots,k\}$.
 \end{itemize}
\end{proposition}

Theorem \ref{T:main} may now be proved analogously to Theorem 1.1 of \cite{HerMcNStr11}. For the convenience of the reader, a sketch of the proof is given here; for details see Section 3 in \cite{HerMcNStr11}. The proof of Proposition \ref{P:antiderivative} is given in Section \ref{SS:antiderivative} below.

\begin{proof}[Proof of Theorem \ref{T:main}]
   Let $f\in\mathcal{C}^{\infty}(\overline{\Omega})$ so that $Pf\in h^{k}(\Omega)$. It follows from Corollary \ref{C:dualityredux} that it suffices to consider $|(f,h)|$ for all $h\in h^{k}(\Omega)$ contained in the unit ball of $h^{-k}(\Omega)$ to estimate $\|Pf\|_{k}$. For $\mathcal{T}=\{T_{j}\}_{j=1}^{m}$ given, choose $U$ and $\zeta$ as in Proposition \ref{P:antiderivative} and write
   $$|(f,h)|\leq |(f,\zeta h)|+|(f,(1-\zeta)h)|.$$
   Since $1-\zeta$ is identically $0$ near $b\Omega$, it follows from Cauchy--Schwarz  inequality and \eqref{E:harmnegnorm} that
   \begin{align}\label{E:estimatedrin}
     |(f,(1-\zeta)h)|\lesssim\|f\|\cdot\|h\|_{-k}\leq \|f\|.
   \end{align}
   Furthermore, (1) of Proposition \ref{P:antiderivative}  yields 
   $$(f,\zeta h)=\Bigl(f,\sum_{\ell=0}^{k}\sum_{J_{\ell}}T_{J_{\ell}}\Bigl(\mathcal{H}_{J_{\ell}}^{k}\Bigr)\Bigr).$$ 
   Integrate by parts repeatedly and then use the Cauchy--Schwarz inequality to obtain
   \begin{align}\label{E:estimaterand}
     |(f,\zeta h)|\lesssim \sum_{\ell=0}^{k}\sum_{J_{\ell}}\|T_{J_{\ell}}f\|\cdot\|\mathcal{H}_{J_{\ell}}^{k}\|\lesssim
     \|f\|_{k,\mathcal{T}}\cdot\|h\|_{-k}\leq \|f\|_{k,\mathcal{T}},
  \end{align}   
  where (2) of Proposition \ref{P:antiderivative} was used as well as that $\|h\|_{-k}\leq 1$. Inequalities \eqref{E:estimatedrin} and \eqref{E:estimaterand}, together with Corollary \ref{C:dualityredux} imply that \eqref{E:main} holds for all $f\in\mathcal{C}^{\infty}(\overline{\Omega})$; removing the smoothness assumptions on $f$ can be done analogously to Lemma 4.2 in \cite{HerMcN10}. 
\end{proof}

\medskip

\subsection{Proof of Proposition \ref{P:antiderivative}}\label{SS:antiderivative}

The proof of Proposition \ref{P:antiderivative} is done in several steps -- the first one essentially consists of constructing the $k$-th antiderivative  of a given function along the integral curves of a normal vector field near the boundary as follows.

\begin{lemma}\label{L:antiderivative}
 Let $\Omega\Subset\mathbb{R}^{n}$, $n\geq 2$, be a smoothly bounded domain, and $N$ a smooth vector field on $\Omega$ which is transversal to $b\Omega$.  Then there exists a neighborhood $U$ of $b\Omega$ such that for a given function $a\in\mathcal{C}^{\infty}(U)$ and $k\in\mathbb{N}$
 \begin{align}\label{E:kantiderivative}
   g(x)=\left(\mathfrak{A}^{2}\circ(N^{2}-a\triangle) \right)^{k}[g](x)+\sum_{\ell=0}^{k-1}\left(\mathfrak{A}^{2}\circ(N^{2}-a\triangle) \right)^{\ell}\circ
   \mathfrak{A}^{2}[a\triangle g](x)
 \end{align}
 holds for all $g\in\mathcal{C}_{\overline{U}}^{\infty}(\Omega)$ and $x\in\Omega$.
\end{lemma}
\begin{proof}
   Let $U$ be the neighborhood of $b\Omega$ and $\varphi$ be flow map associated to $N$ defined on $(-1,1)\times U$ as described at the beginning of Section \ref{SS:Flowstuff}. For $k=1$ first  apply \eqref{E:FTC} to $g$ and then to $Ng$ to obtain
   \begin{align}\label{E:1antiderivative}
       g(x)=\mathfrak{A}[Ng](x)&=\mathfrak{A}^{2}[N^{2}g](x)\notag\\
       &=\mathfrak{A}^{2}\left[(N^{2}-a\triangle)g\right](x)+\mathfrak{A}^{2}[a\triangle g](x)
   \end{align}
   for $x\in\Omega$.
   The general case now follows by induction. That is, suppose \eqref{E:kantiderivative} holds for a given $k\in\mathbb{N}$ and replace the first $g$ on the right hand side of \eqref{E:kantiderivative} with the term on the right hand side of \eqref{E:1antiderivative}.
\end{proof}

To deal with terms of the form $\left(\mathfrak{A}^{2}\circ(N^{2}-a\triangle) \right)^{k}$ (for some function $a$) the ellipticity of the Laplace operator comes into play.
The latter property lets us, e.g., when $k=1$, replace one $N$-derivative in the $N^{2}$-term  by a linear combination of tangential derivatives, which then are commuted to the outside.

\begin{lemma}\label{L:A2Tcommute}
   Suppose the hypotheses of Lemma \ref{L:antiderivative} hold. Let $\mathcal{T}=\{T_{j}\}_{j=1}^{m}$ be a tangential spanning set for $\Omega$. Then there exist a neighborhood $U$ of $b\Omega$ and a function $a\in \mathcal{C}^{\infty}(U)$ such that for any $k\in\mathbb{N}$ there exist operators $G_{J_{\ell}}^{k}$, $\ell\in\{0,\dots,k\}$, belonging to the class of operators $\sum_{i=0}^{k-\ell}
   \sum_{\alpha\in\mathbb{N}_{0}^{2k},|\alpha|\geq i}\mathcal{A}_{\alpha,k+i}^{2k}$ such that
   \begin{align*}
     \left(\mathfrak{A}^{2}\circ(N^{2}-a\triangle) \right)^{k}=\sum_{\ell=0}^{k}\sum_{J_{\ell}}T_{J_{\ell}}\circ G_{J_{\ell}}^{k}.
   \end{align*}
\end{lemma}
\begin{proof}
Let $N$  be the smooth vector field on $\overline{\Omega}$ which is transversal to $b\Omega$.  Then there exists a neighborhood $U$ of $b\Omega$ such that the span of $T_{1}(x),\dots,T_{m}(x)$ and $N(x)$ is $\mathbb{R}^{n}$ for any $x\in U$. That is, for any $k\in\{1,\dots,n\}$ there exist functions $a_{k}^{j}\in\mathcal{C}^{\infty}(U)$, $j\in\{0,\dots,m\}$, such that
 \begin{align}\label{E:Ntransversal} 
  \frac{\partial}{\partial x_{k}}=a_{k}^{0}N+\sum_{j=1}^{m}a_{k}^{j}T_{j}.
\end{align}  
   The transversality of $N$ to $b\Omega$ implies that for each $x\in b\Omega$ there exists a $k\in\{1,\dots,n\}$ such that $a_{k}^{0}(x)\neq 0$. In fact, after possibly shrinking the neighborhood $U$, it may be assumed that $a^{-1}:=\sum_{k=1}^{n}(a_{k}^{0})^{2}>0$ on $U$.
Note that  \eqref{E:Ntransversal} implies that
\begin{align*}
  \frac{\partial^{2}}{\partial x_{k}^{2}}=(a_{k}^{0})^{2}N^{2}+\sum_{j=1}^{m}T_{j}X_{jk}+X_{0k},
\end{align*}
where the $X_{jk}$'s are smooth differential operators of order $1$. Summing over $k$ then leads to
\begin{align}\label{E:LaplaceN2}
  N^{2}-a\triangle=\sum_{j=1}^{m}T_{j}Y_{j}+Y_{0}\qquad\sjump\text{on}\sjump U
\end{align}
for some smooth differential operators, $Y_{j}$, $j\in\{0,\dots,m\}$, of order $1$.  After possibly shrinking the neighborhood $U$ of $b\Omega$, it may be assumed that the flow map $\varphi$ associated to $N$ is defined on $(-1,1)\times U$ as described in Section \ref{SS:Flowstuff}, and hence the setting portrayed therein applies here.

\medskip

Therefore, it needs to be proved that  
\begin{align}\label{E:A2Tcommute}
   \left(\mathfrak{A}^{2}\circ\Bigl(\sum_{j=1}^{m}T_{j}Y_{j}+Y_{0}\Bigr)\right)^{k}=\sum_{\ell=0}^{k}\sum_{J_{\ell}}T_{J_{\ell}}\circ G_{J_{\ell}}^{k}
\end{align}   
    holds for some  $G_{J_{\ell}}^{k}\in\sum_{i=0}^{k-\ell}
   \sum_{\alpha\in\mathbb{N}_{0}^{2k},|\alpha|\geq i}\mathcal{A}_{\alpha,k+i}^{2k}$ , which will be done by induction on $k\in\mathbb{N}$.
The case $k=1$ follows easily from commuting $\mathfrak{A}^{2}$ by $T_{j}$, $j\in\{1,\dots,m\}$. That is, setting
$G_{j}^{1}=\mathfrak{A}^{2}Y_{j}$ for $j\in\{1,\dots,m\}$,
and $G_{0}^{1}=\sum_{j=1}^{m}[\mathfrak{A}^{2},T_{j}]Y_{j}  + \mathfrak{A}^{2}Y_{0}$ yields \eqref{E:A2Tcommute} for $k=1$. Moreover, clearly both $G_{j}^{1}$,   $j\in\{1,\dots,m\}$, and $\mathfrak{A}^{2}Y_{0}$ belong to $\in\mathcal{A}^{2}_{0,1}$. Also, it follows from part (i) of  Lemma \ref{L:ADcommutator} that $$\sum_{j=1}^{m}[\mathfrak{A}^{2},T_{j}]Y_{j}\in\mathcal{A}_{0,1}^{2}+\sum_{i=1}^{2}\mathcal{A}^{2}_{e_{i},2},$$ which concludes the proof for $k=1$.

Next, let $k\in\mathbb{N}$ be given and suppose that \eqref{E:A2Tcommute} holds for some $G_{J_{\ell}}^{k}\in\sum_{i=0}^{k-\ell}
   \sum_{\alpha\in\mathbb{N}_{0}^{2k}|\alpha|\geq i}\mathcal{A}_{\alpha,k+i}^{2k}$. Then commuting $G_{J_{\ell}}^{k}\circ\mathfrak{A}^{2}$ by $T_{j}$ after using the induction hypothesis yields
 \begin{align*}
  & \left(\mathfrak{A}^{2}\circ\Bigl(\sum_{j=1}^{m}T_{j}Y_{j}+Y_{0} \Bigr) \right)^{k+1}
   =\sum_{\ell=0}^{k}\sum_{J_{\ell}}T_{J_{\ell}}\circ G_{J_{\ell}}^{k}\circ\mathfrak{A}^{2}\circ\Bigl(\sum_{j=1}^{m}T_{j}Y_{j}+Y_{0} \Bigr)\\
  &=\sum_{\ell=0}^{k}\sum_{J_{\ell}}\left\{
  \sum_{j=1}^{m}\left(T_{J_{\ell}}\circ T_{j}(G_{J_{\ell}}^{k}\circ\mathfrak{A}^{2}\circ Y_{j})+T_{J_{\ell}}\circ[G_{J_{\ell}}^{k}\circ\mathfrak{A}^{2},T_{j}]Y_{j}
  \right)+T_{J_{\ell}}\circ G_{J_{\ell}}^{k}\circ\mathfrak{A}^{2}\circ Y_{0}
  \right\}.
 \end{align*}  
 To show that the three operators, $G_{J_{\ell}}^{k}\circ\mathfrak{A}^{2}\circ Y_{j}$, $[G_{J_{\ell}}^{k}\circ\mathfrak{A}^{2},T_{j}]Y_{j}$ and $G_{J_{\ell}}^{k}\circ\mathfrak{A}^{2}\circ Y_{0}$, are in the claimed spaces, note first that part (ii) of Lemma \ref{L:ADcommutator} yields
 $$G_{J_{\ell}}^{k}\circ\mathfrak{A}^{2}\circ Y_{j}\in\sum_{i=0}^{(k+1)-(\ell+1)}\sum_{\alpha\in\mathbb{N}_{0}^{2k},|\alpha|\geq i}\mathcal{A}_{(\alpha,0,0),k+1+i}^{2k+2}\,.$$
 Similarly, one obtains from part (ii) of Lemma \ref{L:ADcommutator} that
 $$G_{J_{\ell}}^{k}\circ\mathfrak{A}^{2}\circ Y_{0}\in \sum_{i=0}^{k-\ell}
\sum_{\alpha\in\mathbb{N}_{0}^{2k},|\alpha|\geq i}\mathcal{A}^{2k+2}_{(\alpha,0,0),k+1+i}\subset\sum_{i=0}^{k+1-\ell}
\sum_{\alpha\in\mathbb{N}_{0}^{2k+2},|\alpha|\geq i}\mathcal{A}^{2k+2}_{\alpha,k+1+i}.$$
 Lastly, for the second term both parts of Lemma \ref{L:ADcommutator} need to be used to obtain
 \begin{align*}
 [G_{J_{\ell}}^{k}\circ\mathfrak{A}^{2},T_{j}]Y_{j}\in \sum_{i=0}^{k-\ell}\sum_{\alpha\in\mathbb{N}_{0}^{2k},|\alpha|\geq i}\mathcal{A}_{(\alpha,0,0),k+1+i}^{2k+2}+&\sum_{i=0}^{k-\ell}\sum_{\alpha\in\mathbb{N}_{0}^{2k},|\alpha|\geq i}\sum_{j=1}^{2k+2}\mathcal{A}_{(\alpha,0,0)+e_{j},k+2+i}^{2k+2}\\
 &\subset\sum_{i=0}^{k+1-\ell}\sum_{\alpha\in\mathbb{N}_{0}^{2k+2},|\alpha|\geq i}\mathcal{A}^{2k+2}_{\alpha,k+1+i},
 \end{align*}
 which concludes the proof.
 \end{proof}
 
 \medskip

Having Lemmas \ref{L:antiderivative} and \ref{L:A2Tcommute} in hand, Proposition \ref{P:antiderivative} may now be proven.
\begin{proof}[Proof of Proposition \ref{P:antiderivative}]
Let $U$ be the neighborhood of $b\Omega$ which is provided by Lemma~\ref{L:A2Tcommute}. Let $\zeta\in\mathcal{C}_{0}^{\infty}(\overline{\Omega}\cap U)$ be a function which equals $1$ in a neighborhood $V\Subset U$ of $b\Omega$. Let $h\in h^{0}(\Omega)$, then it follows from Lemmas \ref{L:antiderivative} and \ref{L:A2Tcommute} that
\begin{align*}
  \zeta h=\sum_{\ell=0}^{k}\sum_{J_{\ell}}T_{J_{\ell}}\circ G_{J_{\ell}}^{k}(\zeta h)+\sum_{\ell=0}^{k-1}T_{J_{\ell}}\circ G_{J_{\ell}}^{k-1}\circ\mathfrak{A}^{2}[a\triangle(\zeta h)],
\end{align*}
where $G_{J_{\ell}}^{j}\in\sum_{i=0}^{j-\ell}\sum_{|\alpha|\geq i}\mathcal{A}_{\alpha,j+i}^{2j}$ for $j=k,k-1$. 
Thus, it remains to be shown that both $G_{J_{\ell}}^{k}(\zeta h) $ and $G_{J_{\ell}}^{k-1}\circ\mathfrak{A}^{2}[a\triangle(\zeta h)]$ belong to $H^{k}(\Omega)$ as well as that their $L^{2}$-norms on $\Omega$ are bounded by $\|h\|_{-k}$ (up to a multiplicative, uniform constant). Note first that \eqref{L:AS} implies that
\begin{align*}
  \left\|G_{J_{\ell}}^{k}(\zeta h)\right\|\lesssim\sum_{i=0}^{k-\ell}\sum_{|\alpha|\geq i}\sum_{|\beta|\leq i+k}
  \left\|t_{x}^{2k+|\alpha|}D^{\beta}(\zeta h)\right\|\lesssim \|h\|_{-k},
\end{align*}
where the last estimate follows from \eqref{E:harmnegnorm}. Similarly, using additionally that $\triangle(\zeta h)=(\triangle\zeta)\cdot h+\langle\nabla\zeta,\nabla h\rangle$, it follows that
\begin{align*}
  \left\|G_{J_{\ell}}^{k-1}\circ\mathfrak{A}^{2}[a\triangle(\zeta h)] \right\|&\lesssim
  \sum_{i=0}^{k-1-\ell}\sum_{|\alpha|\geq i}\sum_{|\beta|\leq i+k-1}\left\|t_{x}^{2k-2+|\alpha|}D^{\beta} \left(\mathfrak{A}^{2}[a\triangle(\zeta h)] \right)\right\|\\
  &\lesssim
  \sum_{i=0}^{k-1-\ell}\sum_{|\alpha|\geq i}\sum_{|\beta|\leq i+k-1}\left\|t_{x}^{2k+|\alpha|}D^{\beta+1} \left(\zeta_{0} h\right)\right\|
\end{align*}
for some $\zeta_{0}\in\mathcal{C}_{0}(\Omega\cap U)$. Hence, with \eqref{E:harmnegnorm} it follows that
\begin{align*}
   \left\|G_{J_{\ell}}^{k-1}\circ\mathfrak{A}^{2}[a\triangle(\zeta h)] \right\|\lesssim \|h\|_{-k}.
\end{align*}
Using analogous arguments and part (ii) of Lemma \ref{L:ADcommutator} also implies that the Sobolev-$k$-norms of $G_{J_{\ell}}^{k}(\zeta h) $ and $G_{J_{\ell}}^{k-1}\circ\mathfrak{A}^{2}[a\triangle(\zeta h)]$ are bounded by the $L^{2}$-norm of $h$ (up to a multiplicative constant). This concludes the proof of Proposition \ref{P:antiderivative}.

 \end{proof}


\section{The harmonic Bergman projection on the unit ball}\label{S:Ball} 
\subsection{Theorem \ref{T:main} on the unit ball}
Let $n\geq 2$. Consider the unit ball $$\mathbb{B}^n=\Bigl\{x\in\mathbb{R}^n: r(x)=\sum_{j=1}^{n}x_j^2-1<0\Bigr\}.$$   The harmonic Bergman projection $P$ on $L^2(\mathbb{B}^n)$ is given by
\begin{align}\label{E:harmB1}
  (Pf)(x)=\int_{\mathbb{B}^n}P(x,y)f(y)\;dV(y)\qquad\sjump\forall\sjump x\in\mathbb{B}^{n},
\end{align}  
where
\begin{align}\label{E:harmB2}
P(x,y)=\frac{1}{nV(\mathbb{B}^n)\left(1-2\langle x,y\rangle+|x|^2|y|^2\right)^{n/2}}
\left(\frac{n(1-|x|^2|y|^2)^2}{1-2\langle x,y\rangle+|x|^2|y|^2}-4|x|^2|y|^2 \right),
\end{align}
see, e.g., Theorem 8.13 in \cite{Axleretal}, or cf. to Section 2 in \cite{Stroethoff98} for a different derivation of \eqref{E:harmB2}.

\medskip

For $1\leq i<j\leq n$, define smooth vector fields $$T_x^{i,j}=x_i\frac{\partial}{\partial x_j}-x_j\frac{\partial}{\partial x_i}.$$ Note that each $T_x^{i,j}$ is tangent to $b\mathbb{B}^n$ at $x\in b\mathbb{B}^n$ since
$T_x^{i,j}(r(x))=0$. In fact, the span of $T_{x}^{i,j}$, $1\leq i<j\leq n$, is the tangent space to $b\mathbb{B}^n$ at $x\in b\mathbb{B}^n$. Hence, 
$\mathcal{T}=\{T_x^{i,j}\}$, $1\leq i<j\leq n$, is a spanning set for $\mathbb{B}^n$. 
Furthermore, 
\begin{align*}
  T_{x}^{i,j}(\langle x, y\rangle)=x_{i}y_{j}-y_{i}x_{j}=-T_{y}^{i,j}(\langle x,y\rangle)
\end{align*}
and  $T_{x}^{i,j}(|x|^2|y|^2)=0=-T_{y}^{i,j}(|x|^2|y|^2)$, so that
\begin{align*}  
  T_{x}^{i,j}\left(P(x,y)\right)=-T_{y}^{i,j}\left(P(x,y)\right)
\end{align*}
holds. Hence, for $f\in\mathcal{C}^{\infty}(\overline{\mathbb{B}^n})$
\begin{align*}
  T_{x}^{i,j}(Pf)(x)&=-\int_{\mathbb{B}^n}T_{y}^{i,j}\left(P(x,y)\right)f(y)\;dV(y)\\
  &=\int_{\mathbb{B}^n}P(x,y)T_{y}^{i,j}\left(f(y)\right)\;dV(y)=P(T^{i,j}f)(x),
\end{align*}
where the last line follows from the fact that the $L^2(\mathbb{B}^{n})$-adjoint of $-T_{y}^{i,j}$ is
$T_{y}^{i,j}$. Thus $[T^{i,j},P]=0$ on $\mathcal{C}^{\infty}(\overline{\mathbb{B}^n})$ for all $1\leq i<j\leq n$. It then follows from 
\eqref{E:harmtang} and the $L^{2}$-boundedness of $P$ that  
\begin{align*}
\|Pf\|_{1}&\lesssim\|Pf\|_{1,\mathcal{T}}+\|f\|\\
&=\sum_{1\leq i<j\leq n}\|[T^{i,j},P]f\|+\|f\|_{1,\mathcal{T}}
\lesssim\|f\|_{1,\mathcal{T}}
\end{align*}
holds for all $f\in H_{\mathcal{T}}^{1}(\mathbb{B}^{n})$. To obtain Theorem \ref{T:main} on the unit ball for  general $k$, repeat the above argument.

\medskip

\subsection{Non-smoothing in tangential direction on the unit disk}
Consider the unit disk $\mathbb{D}=\mathbb{B}^{2}$ and the vector field $N=x_1\frac{\partial}{\partial x_1}+x_2\frac{\partial}{\partial x_2}$, which is normal to $b\mathbb{D}$ at each boundary point $(x_1,x_2)$. To show that the right hand side of \eqref{E:main} may not be substituted by a Sobolev norm only measuring normal derivatives, a sequence of functions $f_k\in L^2(\mathbb{D})$ is constructed, which satisfies
\begin{itemize}
  \item[(i)] $Nf_k\in L^2(\mathbb{D})$ for all $k\in\mathbb{N}$,
  \item[(ii)] there exists no constant $C>0$ such that
  \begin{align*}
    \left\|Pf_{k}\right\|_{1}\leq C\left(\left\|Nf_k\right\|+\|f_k\| \right)\qquad\sjump\forall\sjump k\in\mathbb{N}.
  \end{align*}
\end{itemize}
Define $g_{k}(x)=\frac{(x_1+ix_2)^{k+1}}{|x|^k}$ for $x\in\mathbb{D}\setminus\{0\}$ and $g_{k}(0)=0$ for all $k\in\mathbb{N}$. Then set $f_{k}(x)=\text{Re}(g_{k}(x))$ and note that $f_k\in L^2(\mathbb{D})$ since
\begin{align*}
  \left\|f_k\right\|^2=\int_{0}^{2\pi}\int_{0}^{1}\left|\text{Re}(e^{i(k+1)\theta})\right|^2 r^3\;dr\;d\theta
  =\frac{1}{4}\int_{0}^{2\pi}\cos^{2}\left((k+1)\theta \right)\;d\theta=\frac{\pi}{4}.
\end{align*}
To see that (i) holds, observe that in polar coordinates $(r,\theta)$, $N$ corresponds to the vector field $r\frac{\partial}{\partial r}$ and $f_k$ to the function $r\cdot\cos((k+1)\theta)$, so that $Nf_k=f_k\in L^2(\mathbb{D})$ with uniform norm $\sqrt{\pi}/2$. 

To prove (ii), it suffices to show that there is no constant $C>0$
such that $\|TPf_k\|_1\leq C$ for $T=T^{1,2}=x_1\frac{\partial}{\partial x_2}-x_2\frac{\partial}{\partial x_1}$. It follows from the first part of this section that $T(Pf_k)=P(Tf_k)$. Moreover, that $Tf_{k}=-(k+1)f_{k}$ is easily seen after noticing that $T$ is $-\frac{\partial}{\partial\theta}$ in polar coordinates. Hence (ii) would follow from the sequence $\{(k+1)Pf_k\}_{k}$ not being uniformly bounded in $L^2(\mathbb{D})$.

To compute $Pf_k$ it is more convenient to use a representation of the harmonic Bergman projection in terms of the analytic Bergman projection rather than \eqref{E:harmB1} and \eqref{E:harmB2}. For that write $z=x_1+ix_2$ and $w=y_1+iy_2$. Then, using the identity $1-2\langle x,y\rangle+|x|^2|y|^2=|1-z\overline{w}|^2$, it follows from a straightforward computation that  \eqref{E:harmB2} becomes
\begin{align*}
  P(z,w)=\frac{1}{\pi}\left(
  \frac{1}{(1-z\bar{w})^2}+\frac{1}{(1-\bar{z}w)^2}-1
  \right).
\end{align*}
Note that $\frac{1}{\pi}\frac{1}{(1-z\bar{w})^{2}}$ is the kernel of the analytic Bergman projection $B$ on the unit disk, see, e.g., Proposition 1.4.24 in \cite{Krantz_scv_book}. Thus, for  real-valued functions $f\in L^2(\mathbb{D})$ it follows that
\begin{align}\label{E:harmB3}
  (Pf)(z)=Bf(z)+\overline{Bf(z)}-Bf(0).
\end{align}
Hence, the $L^2(\mathbb{D})$- norm of $2\text{Re}(Bf_k)(z)$ as well as $Bf_{k}(0)$ need to be found. First compute 
\begin{align*}
  (Bg_k)(z)=\frac{1}{\pi}\int_{\mathbb{D}}\frac{w^{k+1}/|w|^k}{(1-z\bar{w})^2}\;dV(w)
  &=\frac{1}{\pi}\sum_{m=0}^{\infty}(m+1)\int_{\mathbb{D}}\frac{w^{k+1}}{|w|^k}z^m\bar{w}^m\;dV(w)\\
  &=2\frac{k+2}{k+4}z^{k+1},
\end{align*}
where polar coordinates were used to compute the integral. Then note that an analog computation yields that $B\overline{g_{k}}$ is zero. Thus, $$Pf_k(z)=4\frac{k+2}{k+4}\text{Re}(z^{k+1}).$$ Since $\|\text{Re}(z^{k+1})\|=\frac{\sqrt{\pi}}{\sqrt{2k+4}}$, it may be concluded that
\begin{align*}
  \left\|Pf_k\right\|=2\sqrt{2\pi}\frac{\sqrt{k+2}}{(k+4)},
\end{align*}
and hence $\|TPf_{k}\|=(k+1)\|Pf_{k}\|$ is not uniformly bounded.

\medskip

\begin{remark}
 It was shown in Section 5 of \cite{HerMcN10} that the analytic Bergman projection $B$ does not exhibit smoothing in the tangent direction on the upper half plane $\mathbb{H}$. Since $B\circ P=B$, it then follows that such smoothing cannot hold for $P$ on $\mathbb{H}$ either. Nevertheless, the above example is included here as it is more feasible and illustrative than the one given in \cite{HerMcN10}.
\end{remark}

\bibliographystyle{plain}
\bibliography{Her12}
\end{document}